\documentclass[10 pt, reqno]{amsart}
\usepackage[utf8]{inputenc}
\usepackage{todonotes}
\usepackage{new-style}
\newcommand{\Z}{\mathbb{Z}}

\usepackage[backend=biber, sorting=nyt, maxnames=100,backref=true, giveninits=true, sortcites=true]{biblatex}
\usepackage{tikz}
\usetikzlibrary{decorations.pathreplacing,calc}

\newcommand{\dist}{\operatorname{dist}}

\title{Long Directed Cycles in Vertex-Transitive Digraphs}

\author{Bowen Li$^{\star}$}
\address{$^{\star}$Department of Mathematics, University of Illinois Urbana--Champaign}
\email{bowenl6@illinois.edu}
\thanks{Bowen Li is supported by NSF grant RTG DMS-1937241, UIUC Campus Research Board RB24012, and a University Block Grant Fellowship.}

\author{Abhishek Methuku$^{\star}$}
\email{methuku@illinois.edu}
\thanks{Abhishek Methuku is supported by the UIUC Campus Research Board Award RB25050 and Simons Foundation Award SFI-MPS-TSM-00025583.}


\date{}

\begin{document}

\begin{abstract}
The search for Hamiltonian cycles in vertex-transitive graphs and digraphs is a classical problem at the interface of graph theory and group theory. In the undirected setting, this goes back to the well-known conjectures of Lov\'asz and Thomassen concerning Hamiltonian paths and cycles in connected vertex-transitive graphs. Dating back to Rankin's 1946 work, the directed analogue has an even longer history, linking the search for long cycles to classical group-rearrangement problems. 

Trotter and Erd\H{o}s showed in 1978 that connected vertex-transitive digraphs need not be Hamiltonian. In light of this result, Alspach asked in 1981 whether there exist connected vertex-transitive digraphs whose longest directed cycle misses arbitrarily many vertices. This question was only recently resolved by Buci\'c, Hendrey, Mohar, Steiner and Yepremyan, who constructed connected vertex-transitive digraphs on $n$ vertices whose longest directed cycle omits $(1-o(1))\log n$ vertices. They conjectured that the number of omitted vertices can grow linearly with $n$, remarking that it would already be interesting to improve their logarithmic lower bound to a polynomial bound. In this paper, we confirm their conjecture in a strong form by constructing infinitely many connected vertex-transitive digraphs on $n$ vertices whose longest directed cycle omits at least $n/12$ vertices.

In the same work, Buci\'c, Hendrey, Mohar, Steiner and Yepremyan also proved that every connected vertex-transitive digraph on $n$ vertices contains a directed cycle of length $\Omega(n^{1/3})$, giving the first lower bound for this problem that grows with $n$. We improve this to $\Omega(\sqrt n)$, matching the order of Babai's classical theorem from 1979 for undirected vertex-transitive graphs.

\end{abstract}

\maketitle

\section{Introduction}

Finding long paths and cycles under structural assumptions is a classical
direction in graph theory. A cycle containing every vertex is called a
\emph{Hamiltonian cycle}, and a graph containing one is said to be
\emph{Hamiltonian}. One of the most intriguing such structural assumptions
is symmetry. Lov\'asz conjectured in 1969 that every connected
vertex-transitive graph contains a Hamiltonian path~\cite{lovasz1970}, and
Thomassen later conjectured that every sufficiently large connected
vertex-transitive graph is Hamiltonian; see~\cite{babai1979}. These
conjectures have attracted considerable
attention~\cite{alspach1981,curran1996,pak2009,witte1984}, but remain widely open,
and most known results require additional assumptions.

On the other hand, Babai~\cite{babai1979} initiated a very general
direction of attack in 1979: instead of asking for a Hamiltonian cycle, one may ask
how long a cycle must exist in every connected vertex-transitive graph
without any further assumptions. He proved that every connected
vertex-transitive graph on $n$ vertices contains a cycle of length
$\Omega(\sqrt n)$. In recent years, several works have improved this result
and connected it to a number of other graph-theoretic
problems~\cite{devos2023,groenland2025,ma2025,norin2025,bucic2026towards}.
The current state-of-the-art bound is $\Omega(n^{2/3-o(1)})$, due to
Buci\'c, Christoph, Pokrovskiy and Steiner~\cite{bucic2026towards}.

The directed analogue of this problem asks the following fundamental question: how long a directed cycle must exist in an $n$-vertex connected vertex-transitive digraph? This directed variant possesses an even richer history. In the most important special case of Cayley digraphs, the directed setting is intrinsically more natural from a group-theoretic perspective, as it allows for arbitrary generating sets rather than strictly symmetric ones. In this Cayley digraph setting, the search for directed cycles translates directly into a natural group rearrangement problem. Indeed, the earliest investigations of this topic were carried out by Rankin in 1946~\cite{rankin1948} and, independently, by Rapaport-Strasser in 1959~\cite{rapaport1959}. Both began with precisely this group-theoretic question and then recast it in the language of Cayley digraphs. Both works attribute their motivation to campanology, alongside the knight's tour problem in the latter case. Strikingly, various instances of this exact problem had already been solved in practice more than a century prior to these mathematical formalizations (see~\cite[Section~4]{rankin1948} or~\cite{white1984} for historical details).

The directed analogue of Thomassen's conjecture was disproved in 1978 by Trotter and Erd\H{o}s~\cite{trottererdos1978}, who exhibited an infinite family of connected vertex-transitive digraphs without Hamiltonian cycles. Motivated by this result, Alspach asked in 1981 whether such digraphs must at least be ``nearly'' Hamiltonian~\cite{alspach1981}. To formalize this question, following~\cite{BHMSY}, we use the \emph{perimeter gap}---defined as the difference between the number of vertices and the length of the longest directed cycle in a digraph---as a measure of how far a digraph is from being Hamiltonian. In particular, Alspach (Question~7 in~\cite{alspach1981}) asked if, for any constant $C$, there exists a connected vertex-transitive digraph with a perimeter gap larger than $C$.

Recently, Buci\'c, Hendrey, Mohar, Steiner and Yepremyan answered Alspach's question in the affirmative by constructing connected vertex-transitive digraphs whose perimeter gap tends to infinity.

\begin{theorem}[Buci\'c, Hendrey, Mohar, Steiner and Yepremyan~\cite{BHMSY}]
\label{thm:BHMSY-gap}
For infinitely many natural numbers $n$, there exists a connected vertex-transitive digraph on $n$ vertices with perimeter gap at least $(1-o(1))\log n$.
\end{theorem}

They also proposed the following conjecture, remarking that it would already be interesting to improve their logarithmic lower bound to a polynomial one.

\begin{conjecture}[Conjecture~4.1 of~\cite{BHMSY}]
\label{conj:BHMSY-linear-gap}
There exists an $\varepsilon>0$ such that, for infinitely many values of $n$, there exists a connected vertex-transitive digraph on $n$ vertices whose perimeter gap is at least $\varepsilon n$.
\end{conjecture}

Our first result resolves Conjecture~\ref{conj:BHMSY-linear-gap} in a
strong form as follows.

\begin{theorem}\label{thm:perimeter-gap}
For infinitely many natural numbers $n$, there exists a connected vertex-transitive digraph $D$ on $n$ vertices with perimeter gap at least $n/12$.
\end{theorem}

Thus, the perimeter gap need not merely tend to infinity: it can be linear in the number of vertices.

To prove Theorem~\ref{thm:perimeter-gap}, we introduce a combinatorial parity obstruction based on permutations to exclude Hamiltonicity, in contrast to the number-theoretic obstruction used in~\cite{BHMSY}, which relies on Cartesian products of directed cycles and arithmetic constraints on cycle lengths. 
Roughly speaking, our construction is a \emph{cyclically layered} directed graph $D$ whose vertex set is partitioned into layers, each identified with the same $12$-element set $X$, with edges running only from one layer to the next in a fixed cyclic order. Crucially, $D$ is designed so that any hypothetical Hamiltonian cycle induces a bijection between consecutive layers that acts as an \emph{even permutation} on $X$. By composing these layer-to-layer bijections along one full traversal of the graph, the induced first-return permutation on a fixed layer must also be even. However, for the directed cycle to be Hamiltonian, this induced permutation must act transitively on the $12$ elements of the layer, meaning it must be a cyclic permutation of length $12$. Since any cyclic permutation of even length is an odd permutation, this yields a parity contradiction. A detailed overview of this argument is provided in Section~\ref{proofoverview:improvedupperbound}.

Having established that connected vertex-transitive digraphs can be far
from Hamiltonian, the complementary question becomes even more natural:
how long a directed cycle can one always guarantee? Before the work of
Buci\'c, Hendrey, Mohar, Steiner and Yepremyan~\cite{BHMSY}, it was not even known
whether one can guarantee a directed cycle whose length grows with $n$.
They proved the first such bound, establishing a directed analogue of
Babai's well-known theorem.

\begin{theorem}[Buci\'c, Hendrey, Mohar, Steiner and Yepremyan~\cite{BHMSY}]
\label{thm:BHMSY-cycle}
Every connected vertex-transitive digraph $D$ on $n\ge 2$ vertices contains
a directed cycle of length $\Omega(n^{1/3})$.
\end{theorem}

A fundamental obstacle in translating bounds in the undirected setting to the directed setting is the stark disparity between the maximum lengths of paths and cycles. In undirected graphs, many results guaranteeing the existence of long cycles rely heavily on the structural property that in a $2$-connected graph, any two longest cycles must intersect. However, as noted in~\cite{BHMSY}, this intersection property fails entirely for digraphs. Indeed, there exist families of strongly $2$-connected, $2$-regular digraphs that contain arbitrarily many vertex-disjoint longest directed cycles---which can be restricted to a constant length, such as four---even as the maximum length of a directed path in these same graphs grows unboundedly. 

Consequently, one cannot infer the existence of long directed cycles merely from strong connectivity and long directed paths. Indeed, while Buci\'c, Hendrey, Mohar, Steiner and Yepremyan successfully established an $\Omega(\sqrt n)$ lower bound for directed paths, extracting long directed cycles demanded completely different machinery. By reducing the problem back to the undirected setting via an auxiliary \emph{cycle graph}, they overcame this obstacle, though this reduction ultimately yielded a weaker lower bound of $\Omega(n^{1/3})$ for directed cycles.

Our second result overcomes this barrier, establishing an $\Omega(\sqrt n)$ lower bound for directed cycles that matches Babai's classical bound in the undirected setting. In particular, it shows that the best lower bound known for directed paths extends to directed cycles as well.

\begin{theorem}\label{thm:sqrt-cycle}
Every connected vertex-transitive digraph $D$ on $n\ge 2$ vertices contains
a directed cycle of length $\Omega(\sqrt n)$.
\end{theorem}

Together, Theorems~\ref{thm:perimeter-gap} and~\ref{thm:sqrt-cycle} show that vertex-transitive digraphs can have linear perimeter gap, while every connected vertex-transitive digraph contains a directed cycle of length $\Omega(\sqrt n)$.

Our proof of Theorem~\ref{thm:sqrt-cycle} introduces a weighted version of the cycle graph from~\cite{BHMSY} to eliminate the main loss in their argument. Their argument naturally splits into two regimes based on the directed diameter of $D$. When the diameter is small, weak expansion properties guarantee the existence of a long directed cycle. When the diameter is large, they reduce the directed problem to an undirected one by analyzing the auxiliary cycle graph, whose vertices represent the directed cycles of $D$ and whose edges denote intersections. 

In their unweighted framework, translating paths from the cycle graph back to walks in $D$ introduces a multiplicative factor equal to the circumference, enforcing a quadratic relation between the directed diameter and the circumference of~$D$. As a first step toward eliminating this loss, we weight each vertex in the cycle graph by its cycle length, ensuring that the weighted diameter of the cycle graph upper-bounds the directed diameter of $D$. While this introduces a new challenge---as weighted distances can vastly exceed ordinary graph distances---our key insight is that the cycle graph inherently provides ``weighted shortcuts'' that strictly limit the weight of any geodesic subpath on an induced cycle. On the other hand, a careful weighted adaptation of an elegant strategy from~\cite{BHMSY} allows us to trap a heavy weighted-geodesic subpath inside an induced cycle of the cycle graph, with weight comparable to the weighted diameter. Combining this structural trap with our weighted shortcut insight demonstrates that the weighted diameter of the cycle graph is bounded by a constant multiple of the circumference of $D$. This establishes a linear relation between the directed diameter and the circumference of $D$ that improves the $\Omega(n^{1/3})$ bound of \cite{BHMSY} to $\Omega(\sqrt n)$. We provide a detailed overview of this in Section~\ref{proofoverview:improvedlowerbound}.

\subsection{Structure of the paper}
The rest of the paper is organized as follows. In Section~\ref{proofoverview}, we give an overview of the main ideas behind the proofs of Theorems~\ref{thm:perimeter-gap} and~\ref{thm:sqrt-cycle}. In Section~\ref{sec:weightedcyclemachinery}, we develop the weighted cycle graph machinery, establish a linear relation between the directed diameter and the circumference of a connected vertex-transitive digraph, and prove Theorem~\ref{thm:sqrt-cycle}. Finally, in Section~\ref{sec:perimetergap}, we present our layered construction proving Theorem~\ref{thm:perimeter-gap}.

\subsection{Notation}
Throughout the paper, all graphs and digraphs are finite unless
explicitly stated otherwise. A digraph $D$ is \emph{vertex-transitive} if
for every pair of vertices $u,v\in V(D)$ there exists an automorphism of
$D$ mapping $u$ to $v$.

Given a digraph $D$, we denote its vertex set by $V(D)$ and its arc set
by $A(D)$. For vertices $u,v\in V(D)$, we write $d(u,v)$ for the
directed distance from $u$ to $v$, namely the length of a shortest
directed path from $u$ to $v$, setting $d(u,v)=\infty$ if no such path
exists. The directed diameter of $D$ is
$
\max_{u,v\in V(D)} d(u,v).
$
A digraph is \emph{strongly connected} if its directed diameter is
finite. It is \emph{connected} if its underlying undirected graph is
connected. For vertex-transitive digraphs these notions coincide. The
\emph{circumference} of a digraph $D$, denoted by $\ell$, is the length
of a longest directed cycle in $D$.

For an undirected graph $G$ and a set $S\subseteq V(G)$, we write $N(S)$ for the set of all vertices in $V(G)\setminus S$ having a neighbor in $S$. For convenience, throughout the paper we often identify a subgraph of $G$ with its vertex set. Thus, if $P$ is a path in $G$, we write $N(P)$ instead of $N(V(P))$, and $P \cap S$ for the set of vertices of $P$ that lie in $S$.

\section{Proof Overview}
\label{proofoverview}

\subsection{Proof sketch for Theorem~\ref{thm:perimeter-gap}}
\label{proofoverview:improvedupperbound}

The approach of \cite{BHMSY} guarantees a logarithmic perimeter gap by combining a global divisibility condition on all directed cycle lengths with an obstruction to Hamiltonicity. Specifically, they employ a number-theoretic obstruction---relying on Cartesian products and arithmetic constraints---to forbid a Hamiltonian cycle, which forces the circumference to drop due to the divisibility requirement. The key novelty of our construction is the use of a \emph{combinatorial parity obstruction}, encoded in the local transitions between layers, instead of a number-theoretic obstruction.
More precisely, for a positive integer
$m$ we construct a cyclic layered digraph $D$ on
$$
X\times \mathbb Z_m,
\qquad
X=\{(a,b):a,b\in[4],\ a\neq b\}.
$$
The edges always move from layer $t$ to layer $t+1$, and have the form
$$
(a,b,t)\to (b,c,t+1),
\qquad c\notin\{a,b\}.
$$ 
Thus every directed cycle has length divisible by $m$. The main point is
to show that a cycle of length $12m$ cannot exist. If there were such a
Hamiltonian cycle, then the edges it uses from one layer to the next would
induce a bijection of the $12$-element set $X$. Each of these local
bijections has the same parity: in fact, a simple check shows that it is
always an even permutation of $X$. Now follow the alleged Hamiltonian cycle
for $m$ steps starting from a fixed layer. Since the layer coordinate has
returned to where it started, these $m$ local bijections compose to a
permutation of that same copy of $X$, which must still be even. On the
other hand, because the Hamiltonian cycle is a single cycle through all
$12m$ vertices, its successive returns to the fixed layer must pass
through all $12$ vertices of that layer in one orbit. The resulting
first-return permutation is therefore a $12$-cycle, which is odd. This
parity contradiction rules out Hamiltonian cycles. Because every directed cycle in $D$ has length divisible by $m$, the longest directed cycle can have length at most $11m$. Hence, the perimeter gap is at least
$
m=|V(D)|/12.
$

\subsection{Proof sketch for Theorem~\ref{thm:sqrt-cycle}}
\label{proofoverview:improvedlowerbound}

The proof naturally splits into two regimes according to the directed diameter $d$ of the digraph $D$. When $d$ is small, the expansion properties of vertex-transitive digraphs, established in~\cite{BHMSY}, guarantee a directed cycle of length
$
\ell\ge \frac{n}{9d}.
$
Thus, the main challenge is to obtain a strong lower bound on $\ell$ when the directed diameter is large.

Following~\cite{BHMSY}, we analyze the auxiliary cycle graph $C(D)$, whose vertices represent the directed cycles of $D$, with adjacency corresponding to nonempty intersection. One of our key ideas is to equip the cycle graph $C(D)$ with a weighted metric by assigning each vertex $F$ the weight
$
w(F)=|F|,
$
equal to the length of the corresponding directed cycle in $D$.

To see how this weighting helps overcome the main source of loss in the argument of~\cite{BHMSY}, recall that their analysis uses the usual graph metric on the cycle graph, in which distance is measured by the number of edges in a shortest path. Consequently, a path of length $k$ in $C(D)$ yields a directed walk in $D$ obtained by traversing $k$ directed cycles, each of which may have length as large as the circumference $\ell$. Thus, the resulting walk may have length as large as $k\ell$. Intuitively, this estimate is only sharp if every cycle encountered has length close to $\ell$, but the unweighted metric cannot distinguish this from the much more typical situation where many of the traversed cycles are considerably shorter. This unavoidable loss introduces an extra factor of $\ell$, yielding the square-root relation $\ell \ge \sqrt d$ (equivalently, $\ell \ge d/\ell$) established in~\cite{BHMSY}. In this paper, we improve this bound to a linear relation, proving that $\ell \ge d/300$.

By contrast, the weighted metric assigns to each vertex of the cycle graph exactly the cost incurred by traversing the corresponding directed cycle in $D$. Consequently, the weighted length of a path in the cycle graph upper-bounds the length of the corresponding directed walk in $D$, implying that the weighted diameter of the cycle graph is at least the directed diameter of $D$.

At first sight, however, this merely replaces one difficulty by another: weighted distances can be substantially larger than ordinary graph distances, so the weighted diameter itself could be much larger than the circumference. The heart of our proof is to show that this cannot happen.

Our first ingredient establishes a weighted shortcut property for induced cycles in the cycle graph. Given an induced cycle $H$ in $C(D)$, as in \cite{BHMSY}, we construct a directed cycle $Z$ in $D$ that intersects every directed cycle represented by the vertices of $H$. Because the length of $Z$ cannot exceed the circumference $\ell$, it corresponds to a vertex in $C(D)$ of weight at most $\ell$ that is adjacent to every vertex of $H$. Crucially, this central vertex acts as a low-weight \emph{shortcut}: the path connecting any two vertices on $H$ via $Z$ has a total weight of at most $3\ell$. Consequently, any weighted-geodesic subpath along an induced cycle is bounded by this shortcut, limiting its total weight to $3\ell$.

Our second ingredient adapts an elegant structural approach of~\cite{BHMSY} through an extremal argument tailored to the weighted metric. More precisely, we show that if the weighted diameter is large, the cycle graph must contain an induced cycle featuring a heavy weighted-geodesic subpath. To achieve this, we consider an induced path $P$ of \emph{maximum weight} that terminates in a heavy weighted-geodesic segment $Q$. By exploiting the near-transitivity of the cycle graph, we map another long weighted geodesic so that it extends this terminal segment $Q$. If the shifted geodesic avoided the earlier vertices of $P$, we could concatenate it with $P$ to obtain an even heavier induced path, contradicting the maximality of $P$. Thus, the shifted geodesic must return to $P$, either meeting it or becoming adjacent to it, and this forced return produces an induced cycle capturing a significant fraction of the weighted diameter.

The weighted setting, however, introduces a structural obstacle. The approach of \cite{BHMSY} relies on the fact that the intersection regions arising from rerouting the path incur only a bounded loss in ordinary graph distance. In our framework, a single vertex in such an intersection region can have weight as large as $\ell$, potentially removing a substantial part of the heavy geodesic subpath we are trying to extract. We overcome this by carefully tracking the weights along the shifted paths and proving that the total weight lost during the rerouting step is bounded by an absolute constant multiple of $\ell$. This ensures that the resulting induced cycle successfully retains a weighted geodesic whose weight is a constant fraction of the total weighted diameter. Together with the weighted shortcut property we have established, this implies a contradiction whenever the weighted diameter of the cycle graph is greater than $300\ell$. Consequently, although weighted distances are potentially much larger than ordinary graph distances, the weighted diameter remains bounded by a constant multiple of the circumference.

Since the weighted diameter of the cycle graph $C(D)$ lies between the directed diameter $d$ of $D$ and $300\ell$, we obtain the desired linear relation $\ell \ge d/300$. Combining this with the aforementioned bound $\ell \ge n/(9d)$ from~\cite{BHMSY} yields $\ell=\Omega(\sqrt n)$, completing the proof.

\section{Weighted Cycle Graphs and a Linear Diameter--Circumference Relation}
\label{sec:weightedcyclemachinery}

Throughout this section, $D$ denotes a connected vertex-transitive
digraph on $n$ vertices, with directed diameter $d$ and circumference
$\ell$. The following lemma follows immediately from Lemmas 2.2 and 2.4 of~\cite{BHMSY}, where the authors establish weak vertex expansion that is inversely proportional to the directed diameter and extract the desired cycle via a directed depth-first search (DFS) argument.

\begin{lemma}\label{lem:small-diameter-bound}
Every connected vertex-transitive digraph $D$ on $n$ vertices with
directed diameter $d$ and circumference $\ell$ satisfies
$$
\ell \ge \frac{n}{9d}.
$$
\end{lemma}

The key new ingredient of our proof is the following improved estimate. The argument in \cite{BHMSY} establishes the bound $\ell \ge d/\ell$, resulting in a quadratic dependence $\ell \ge \sqrt{d}$. In this paper, we improve this to a linear relation as follows.

\begin{lemma}\label{lem:diameter-bound}
We have
$$
\ell \ge \frac{d}{300}.
$$
\end{lemma}

Together with Lemma~\ref{lem:small-diameter-bound}, this immediately
implies Theorem~\ref{thm:sqrt-cycle} (see Section~\ref{proofsubsectionputtingeverythingtogether}). 
We make no attempt to optimize the constant factor $300$ in the bound of Lemma~\ref{lem:diameter-bound}.

\subsection{The weighted cycle graph and near-transitivity}

Following~\cite{BHMSY}, the \emph{cycle graph} $C(D)$ is the graph whose vertices are the directed cycles of $D$, with two vertices adjacent exactly when the corresponding directed cycles share a vertex in $D$.

In the rest of this section, let
$
G \coloneqq C(D).
$
For each vertex $F\in V(G)$, corresponding to a directed cycle $F$ of
$D$, define
$
w(F)=|F|.
$
We first record the notions of weighted distance and weighted diameter for a weighted graph.

\begin{definition}
Let $G$ be a graph equipped with a weight function
$
w:V(G)\to\mathbb Z_{>0}.
$
For a path
$
P=x_0x_1\cdots x_k
$
in $G$, define its \emph{weight} by
$
w(P) \coloneqq \sum_{i=0}^{k} w(x_i).
$
For vertices $x,y\in V(G)$, the \emph{weighted distance} $\dist_w(x,y)$ is the minimum value of $w(P)$ taken over all $x$--$y$ paths $P$ in $G$. A path attaining this minimum is called a \emph{weighted geodesic} between $x$ and $y$. Finally, the \emph{weighted diameter} of $G$ is
\[
\Gamma_w(G) \coloneqq \max_{x,y\in V(G)} \dist_w(x,y).
\]
\end{definition}

The next definition introduces the notions of weight-preserving automorphisms and $w$-near transitivity for weighted graphs.

\begin{definition}
Let $G$ be a graph equipped with a weight function
$
w:V(G)\to\mathbb Z_{>0}.
$
An automorphism $\phi$ of $G$ is called
\emph{weight-preserving} if
$
w(\phi(v))=w(v)
$
for every $v\in V(G)$. We say that $G$ is \emph{$w$-nearly transitive} if for every
$x,y\in V(G)$ there exists a weight-preserving automorphism
$\varphi$ of $G$ such that
$
\varphi(x)=y
\text{ or }
\varphi(x)\in N(y).
$
\end{definition}

Before analyzing weighted distances in the cycle graph, we first verify that it inherits the symmetries of the original digraph $D$. The following lemma extends an observation from \cite{BHMSY} by noting that automorphisms of $D$ inherently preserve cycle lengths, ensuring that the induced automorphisms on $C(D)$ are weight-preserving.

\begin{lemma}\label{lem:nearly-vertex-transitive}
The graph $G=C(D)$ is $w$-nearly transitive.
\end{lemma}

\begin{proof}
Any automorphism $\phi$ of $D$ induces a natural automorphism
$\phi'$ of $G=C(D)$. Namely, if $C\in V(G)$ is a directed
cycle of $D$, then $\phi'(C)$ is the image of $C$ under
$\phi$. Since automorphisms map directed cycles to directed
cycles and preserve intersections of cycles, $\phi'$ is indeed
an automorphism of $G$.

Moreover, automorphisms preserve cycle lengths, so for every
$C\in V(G)$ we have
$$
w(\phi'(C))
=
|\phi'(C)|
=
|C|
=
w(C).
$$
Hence $\phi'$ is weight-preserving. Now let $C_1,C_2\in V(G)$. Choose vertices
$
v\in C_1
\text{ and }
u\in C_2.
$
Since $D$ is vertex-transitive, there exists an automorphism
$\phi$ of $D$ such that
$
\phi(v)=u.
$ Then $\phi'(C_1)$ is a directed cycle containing $u$.
Since $u\in C_2$, it follows that
$
\phi'(C_1)\cap C_2\neq\varnothing.
$
Therefore $\phi'(C_1)$ is either equal to $C_2$ or adjacent to
$C_2$ in $G$. Since $\phi'$ is weight-preserving, this shows that $G$ is
$w$-nearly transitive, as desired.
\end{proof}

\subsection{Weighted distance in the cycle graph controls directed diameter}
The following lemma relates the directed diameter of $D$ to weighted distances in the cycle graph. In the unweighted framework of \cite{BHMSY}, converting paths in $C(D)$ to walks in $D$ incurs a multiplicative loss of $\ell$. By instead weighting each cycle by its length, a path's weight in $C(D)$ naturally upper-bounds the corresponding walk length in $D$. When combined with Lemmas~\ref{lem:weight-le-3ell} and~\ref{lem:weighted-near-transitive}, which establish the upper bound $\Gamma_w(G) \le 300 \ell$, it yields the desired linear diameter-circumference relation.

\begin{lemma}\label{lem:delta-ge-d}
We have
\[
    \Gamma_w(G)\geq d.
\]
\end{lemma}

\begin{proof}
By the definition of the directed diameter, there exist vertices $x,y \in V(D)$ such that the shortest directed path from $x$ to $y$ in $D$ has length exactly $d$. Because $D$ is strongly connected, every vertex is contained in at least one directed cycle. Let $A \in V(G)$ be a vertex in the cycle graph $G = C(D)$ corresponding to a cycle containing $x$, and let $B \in V(G)$ be a vertex corresponding to a cycle containing $y$. Let $P = F_0 F_1 \cdots F_k$ be a weighted geodesic path in $G$ from $A$ to $B$, so that $F_0 = A$ and $F_k = B$. By definition, the weighted distance between $A$ and $B$ in $G$ is
$
\dist_w(A,B) = \sum_{i=0}^k |F_i|.
$

We construct a directed walk from $x$ to $y$ in the original digraph $D$ by traversing the cycles in $P$. Starting at $x \in V(F_0)$, follow the directed edges of $F_0$ until reaching a vertex $v_1 \in V(F_0) \cap V(F_1)$. The number of edges traversed in this initial step is at most $|F_0|$. From $v_1$, follow the directed edges of $F_1$ until reaching a vertex $v_2 \in V(F_1) \cap V(F_2)$, forming a directed path of at most $|F_1|$ edges. Iterating this process for each $1 \le i < k$, we travel from $v_i$ along $F_i$ to a vertex $v_{i+1} \in V(F_i) \cap V(F_{i+1})$, traversing at most $|F_i|$ edges. Finally, from $v_k \in V(F_{k-1}) \cap V(F_k)$, we follow the directed edges of $F_k$ until reaching $y \in V(F_k)$, taking at most $|F_k|$ edges.

Concatenating these directed path segments yields a directed walk in $D$ from $x$ to $y$. The total length of this directed walk is at most
$
\sum_{i=0}^k |F_i| = \dist_w(A,B).
$
Since the shortest directed path from $x$ to $y$ in $D$ has length $d$, the length of any directed walk from $x$ to $y$ must be at least $d$. Consequently,
$$
d \le \dist_w(A,B) \le \Gamma_w(G).
$$
This completes the proof of the lemma.
\end{proof}

\subsection{Induced cycles in the cycle graph have short weighted-geodesic subpaths}

The following lemma shows that every induced cycle in the cycle graph yields a directed cycle $Z$ in $D$ whose corresponding vertex in the cycle graph serves as a \emph{low-weight shortcut}. Consequently, every weighted-geodesic subpath of the induced cycle has weight at most $3\ell$. As discussed in the proof overview (Section~\ref{proofoverview}), this bound is crucial for establishing a linear relationship between the directed diameter and the circumference of $D$.

\begin{lemma}\label{lem:weight-le-3ell}
Every weighted-geodesic subpath of an induced cycle in $G=C(D)$ has weight at most $3\ell$.
\end{lemma}

\begin{proof}
Let $H=F_1F_2\cdots F_tF_1$ be an induced cycle in $G=C(D)$, and let $Q$ be a weighted-geodesic subpath of $H$. If $t\leq 3$, then $Q$ contains at most three vertices, each having a weight of at most $\ell$. Thus,
$
w(Q)\leq 3\ell.
$

We may therefore assume that $t\ge 4$. Since $H$ is an induced cycle, the non-consecutive cycles $F_2$ and $F_t$ are not adjacent in $G$, and hence are vertex disjoint. Consequently, $
(F_1\cap F_t)\cap(F_1\cap F_2)=\varnothing$. We may therefore choose vertices
$
v_1\in F_1\cap F_t$
and 
$v_2\in F_1\cap F_2
$
such that the directed path from $v_1$ to $v_2$ along $F_1$ contains no other vertices of $F_t$ or $F_2$.

For each $2\le i\le t-1$, suppose that $v_i\in F_i\cap F_{i-1}$ has already been chosen. Starting from $v_i$, follow the directed edges of $F_i$ until first reaching a vertex of $F_{i+1}$, and denote this vertex by $v_{i+1}$. Finally, from $v_t\in F_t\cap F_{t-1}$, follow the directed edges of $F_t$ until returning to $v_1$. Since $H$ is induced, every cycle $F_i$ intersects only the two neighboring cycles $F_{i-1}$ and $F_{i+1}$. Furthermore, the choice of $v_{i+1}$ as the first vertex of $F_{i+1}$ encountered along $F_i$ ensures that the corresponding subpath of $F_i$ is internally disjoint from $F_{i+1}$. Consequently, these subpaths intersect only at their endpoints, and their concatenation forms a simple directed cycle $Z$ in $D$.

By construction, the cycle $Z$ shares at least one vertex with every cycle $F_i$ in $H$. Because $\ell$ is the circumference of $D$, the length of $Z$ is bounded by $|Z|\leq \ell$. Furthermore, because $H$ is an induced cycle of length $t \geq 4$, $Z$ cannot be identical to any $F_i$; if it were, that $F_i$ would intersect non-adjacent cycles in $H$, directly contradicting the fact that $H$ is induced. Consequently, in the cycle graph $G=C(D)$, $Z$ represents a distinct vertex of weight at most $\ell$ that is adjacent to every $F_i$.

Now let $X$ and $Y$ be the endpoints of the subpath $Q$. Since $X,Y\in\{F_1,\ldots,F_t\}$, there exists a path in $G$ connecting $X$ and $Y$ through $Z$. The total weight of this path is at most
$
w(X)+w(Z)+w(Y)\le 3\ell.
$
Because $Q$ is a weighted geodesic from $X$ to $Y$, its weight cannot exceed the weight of this path. Therefore,
$
w(Q)\le 3\ell,
$
completing the proof of the lemma.
\end{proof}

\subsection{Large weighted diameter gives a long weighted geodesic on an induced cycle}

To deduce the upper bound $\Gamma_w(G) \le 300\ell$, we need to show that the cycle graph contains an induced cycle with a heavy weighted-geodesic subpath. Adapting an elegant argument from~\cite{BHMSY}, the following lemma produces such a cycle while retaining a substantial fraction of the weighted diameter. The main difficulty in the weighted setting is that a single overlapping vertex can have weight as large as $\ell$; the lemma overcomes this by tightly controlling the total weight lost in certain intersection regions. See Section~\ref{proofoverview:improvedlowerbound} for a sketch of the argument.

\begin{lemma}
\label{lem:weighted-near-transitive}
Let $G$ be a connected graph equipped with a positive integer weight function
$
    w:V(G)\to \mathbb{Z}_{>0}.
$
Let
$W=\max_{v\in V(G)} w(v)$
    and
   $\Gamma=\Gamma_w(G)$.
Suppose that $G$ is $w$-nearly transitive. If $\Gamma\geq 300W$, then $G$ contains an induced cycle $H$ with a subpath $Q_H$ such that $Q_H$ is a weighted geodesic and
$$
    w(Q_H)\geq \frac{\Gamma}{2}-100W.
$$
\end{lemma}
\begin{proof}

Choose vertices $v_1, v_2 \in V(G)$ such that
$
    \dist_w(v_1,v_2)=\Gamma,
$
and let $S$ be a weighted geodesic from $v_1$ to $v_2$. Every subpath of $S$ is inherently a weighted geodesic. Consequently, $S$ is an induced path, as any chord would replace a proper subpath with a single edge, bypassing internal vertices of positive weight and thereby shortening the path.

Choose a vertex $m$ on $S$ such that the two subpaths of $S$ from $m$ to its endpoints each have a weight of at least
$
    \Gamma / 2 - W.
$
Denote these subpaths by $L$ and $R$.

Let $\beta \coloneqq \Gamma/2 - 80W$. We will construct another induced path in $G$ as follows: Let $P$ be an induced path of maximum weight subject to the condition that it has a terminal subpath $Q$ (sharing exactly one endpoint with $P$) which is a weighted geodesic satisfying
\begin{equation}
\label{eq:boundweightQ}
w(Q) \ge \beta = \frac{\Gamma}{2} - 80W.
\end{equation}
The collection of induced paths satisfying the above condition is nonempty, since $S$ satisfies it with $Q=R$. Let $u$ be the common endpoint of $Q$ and $P$, let $x$ be the other endpoint of $P$, and let $y$ be the other endpoint of $Q$; see Figure~\ref{figure1}.

Because $G$ is $w$-nearly transitive, there exists a weight-preserving automorphism $\phi$ such that either $\phi(m) = u$ or $\phi(m) \in N(u)$. Let $L'\coloneqq \phi(L)$, $R' \coloneqq \phi(R)$, $S'\coloneqq \phi(S)$, and $u' \coloneqq \phi(m)$. 

\begin{figure}[htbp]
\centering
\begin{tikzpicture}[scale=0.85,>=stealth]

\tikzset{
    vtx/.style={circle,fill=black,inner sep=1.5pt},
    blueedge/.style={blue,thick},
    brace/.style={decorate,decoration={brace,amplitude=5pt}}
}

\coordinate (Xleft) at (-5.8,0);
\coordinate (Y)     at (-2.8,0);
\coordinate (A)     at (-2.0,0);
\coordinate (B)     at (-1.5,0);
\coordinate (U)     at ( 1.0,0);
\coordinate (Up)    at ( 2.0,0);

\coordinate (Top)   at ( 2.0, 2.6);
\coordinate (Ap)    at ( 2.0, 1.3);
\coordinate (Bp)    at ( 2.0,-1.3);
\coordinate (Bot)   at ( 2.0,-2.6);

\draw[thick] (Xleft) -- (Up);
\draw[thick] (Top) -- (Bot);

\node[vtx,label=below:$x$] at (Xleft) {};
\node[vtx,label=above:$y$] at (Y) {};
\node[vtx,label=above:$a$] at (A) {};
\node[vtx,label=above:$b$] at (B) {};
\node[vtx,label=above:$u$] at (U) {};
\node[vtx,label=below right:$u'$] at (Up) {};
\node[vtx,label=right:$a'$] at (Ap) {};
\node[vtx,label=right:$b'$] at (Bp) {};

\node[vtx] at (Top) {};
\node[vtx] at (Bot) {};

\draw[blueedge] (A) -- (Ap);
\draw[blueedge] (B) -- (Bp);
\draw[blueedge] (U) -- (Up);

\draw[brace] ($(U)+(0,-0.25)$) -- ($(Y)+(0,-0.25)$)
    node[midway,below=7pt] {$Q$};

\draw[brace] ($(U)+(0,-1.25)$) -- ($(Xleft)+(0,-1.25)$)
    node[midway,below=7pt] {$P$};
\draw[brace] ($(Top)+(0.75,0)$)--($(Up)+(0.75,0.1)$)
    node[midway,right=7pt] {$L'$};

\draw[brace] ($(Up)+(0.75,-0.1)$) -- ($(Bot)+(0.75,0)$)
    node[midway,right=7pt] {$R'$};

\draw[brace] ($(Top)+(1.65,0)$) -- ($(Bot)+(1.65,0)$)
    node[midway,right=8pt] {$S'$};

\node[align=left] at (-1.2,-3.2)
{
};

\end{tikzpicture}
\caption{Illustration of the paths $P, Q, L', R'$ and $S'$ used in the proof of Lemma~\ref{lem:weighted-near-transitive}. Blue segments highlight vertex pairs whose distance is either zero or one.}
\label{figure1}
\end{figure}
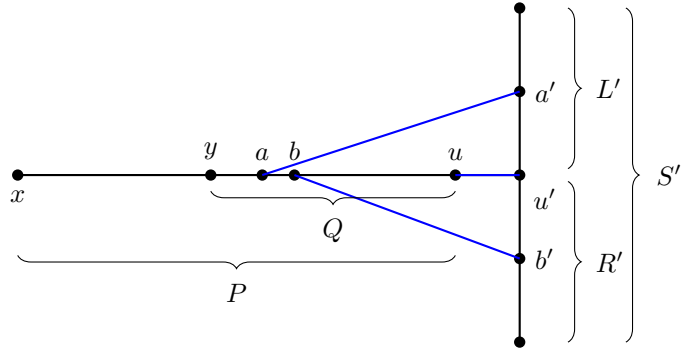

\begin{claim}\label{clm:one-side-near-Q}
Either every vertex of $L'\cap (Q\cup N(Q))$ lies at a weighted distance of at most $10W$ from $u'$ along $S'$, or every vertex of $R'\cap (Q\cup N(Q))$ lies at a weighted distance of at most $10W$ from $u'$ along $S'$.
\end{claim}
\begin{claimproof}
We first establish the following distance bound. Let
$a'\in L'\cap (Q\cup N(Q))$ and $b'\in R'\cap (Q\cup N(Q))$ be arbitrary, and write
$$
A \coloneqq w(S'[a',u'])
\qquad\text{and}\qquad
B \coloneqq w(S'[u',b']).
$$
Then,
$
\min\{A,B\}\le 7W.
$
Choose vertices $a,b\in V(Q)$ such that $a=a'$ or $aa'\in E(G)$, and $b=b'$ or $bb'\in E(G)$. Because $S'$ is a weighted geodesic and $a', b'$ lie on opposite sides of $u'$ in $S'$, we have:
\begin{equation}
\label{dista'b'lowerbound}
\dist_w(a',b') = w(S'[a',b']) = A+B-w(u') \ge A+B-W.
\end{equation}
Conversely, the path that traverses from $a'$ to $a$, moves along $Q$ from $a$ to $b$, and finally steps from $b$ to $b'$ provides an upper bound:
\begin{equation}
\label{dista'b'upperbound}
\dist_w(a',b') \le w(Q[a,b])+4W.
\end{equation}

Let $D_a=w(Q[a,u])$ and $D_b=w(Q[b,u])$ denote the weighted distances along $Q$ from $a$ and $b$, respectively, to the endpoint $u$. Because $a$ and $b$ both lie on $Q$, the subpath between them has weight at most:
\begin{equation}
\label{boundQab}
w(Q[a,b])\le |D_a-D_b|+W. 
\end{equation}

Combining \eqref{dista'b'lowerbound}, \eqref{dista'b'upperbound}, and \eqref{boundQab}, we obtain
\begin{equation}
\label{dista'b'upperbounduseful}
A+B-W \le |D_a-D_b|+5W. 
\end{equation}

Furthermore, since $a=a'$ or $aa'\in E(G)$, and $u=u'$ or $uu'\in E(G)$, the weighted geodesicity of both $Q$ and $S'$ implies
$
|A-D_a|\le w(aa')+w(u'u),
$
where $w(xy)=0$ if $x=y$. As each edge has weight at most $2W$, it follows that $|A-D_a|\le 4W$. Similarly,
$
|B-D_b|\le 4W.
$
Combining these two inequalities with \eqref{dista'b'upperbounduseful}, we obtain
$$
A+B-W
\le |D_a-D_b|+5W
\le |D_a-A|+|A-B|+|B-D_b|+5W
\le |A-B|+13W.
$$
Hence,
$
A+B\le |A-B|+14W.
$
Since
$
A+B-|A-B|=2\min\{A,B\},
$
it follows that
\begin{equation}
\min\{A,B\}\le 7W.
\label{eq:pairwise-bound}
\end{equation}

We are now ready to prove the claim. If every vertex of $L'\cap(Q\cup N(Q))$ is at a weighted distance of at most $10W$ from $u'$ along $S'$, then the first conclusion of the claim holds. Otherwise, choose a vertex $a'\in L'\cap(Q\cup N(Q))$ such that $A = w(S'[a',u'])>10W$. Then, for any $b'\in R'\cap(Q\cup N(Q))$, the bound in \eqref{eq:pairwise-bound} implies that $$B = w(S'[u',b'])\le 7W\le 10W.$$ Since $b'$ was arbitrary, each vertex in $R'\cap(Q\cup N(Q))$ is at a weighted distance of at most $10W$ from $u'$ along $S'$. Hence, either $L'$ or $R'$ satisfies the claim.
\end{claimproof}

By Claim~\ref{clm:one-side-near-Q} and the symmetry between $L'$ and $R'$, we may assume without loss of generality that all vertices in $L'\cap(Q\cup N(Q))$ lie at a weighted distance of at most $10W$ from $u'$ along $S'$. Let $z'$ be the vertex in $L'\cap(Q\cup N(Q))$ that achieves the maximum weighted distance $w(L'[u',z'])$ from $u'$ along $L'$; see Figure~\ref{figure2}. Consequently,
\begin{equation}
\label{eq:boundwLu'z'}
w(L'[u',z'])\leq 10W.   
\end{equation}
Let $z \in V(Q)$ be a vertex that is either equal to or adjacent to $z'$, chosen to maximize the weighted distance $w(Q[z,u])$ from $u$ along $Q$. Since $z'$ lies at a weighted distance of at most $10W$ from $u'$, and $u'$ is either equal to or adjacent to $u$, we bound the weight of the subpath $Q[z,u]$ by the weight of the path traversing through $z'$ and $u'$. More precisely, adopting the convention that $w(xy)=0$ if $x=y$, the weighted geodesicity of $Q$ implies:
\begin{equation}
\label{eq:boundweightQzu}
    w(Q[z,u]) \leq w(zz') + w(L'[z', u']) + w(u'u) \overset{\eqref{eq:boundwLu'z'}}{\leq} 20W.  
\end{equation}

Let $Q'\coloneqq Q[y,z]$. Then,
\begin{equation}
    w(Q')\geq w(Q)- w(Q[z,u]) \overset{\eqref{eq:boundweightQzu}}{\ge} w(Q)- 20W \overset{\eqref{eq:boundweightQ}}{\geq} \frac{\Gamma}{2}-100W.
   \label{eq:Q-prime-large}
\end{equation}

Recall that $u'$ is an endpoint of $L'$. Let $v'$ denote the other endpoint of $L'$, and set $L'' \coloneqq L'[z',v']$ (see Figure~\ref{figure2}). Since $w(L) \ge \Gamma/2 - W$ and $\phi$ is a weight-preserving automorphism, we have $w(L') \ge \Gamma/2 - W$. Combining this with the upper bound $w(L'[u',z']) \le 10W$ from \eqref{eq:boundwLu'z'}, and noting that $w(L'') \ge w(L') - w(L'[u',z'])$, we obtain:
\begin{equation}
    w(L'') \ge \frac{\Gamma}{2} - 11W > \beta.
\label{eq:L-double-prime-large}
\end{equation}

\begin{claim}\label{claim:L''P_meets}
   We have $L'' \cap (P[x,y] \cup N(P[x,y])) \ne \emptyset$.
\end{claim}
\begin{claimproof}
Suppose, for contradiction, that $L'' \cap (P[x,y] \cup N(P[x,y])) = \emptyset$. We construct a path $T$ from $x$ to $v'$ by concatenating the paths $P[x,y]$, $Q[y,z]$, and $L''$, including  $zz'$ if $z \neq z'$. Since $z=z'$ or $zz' \in E(G)$, this concatenation forms a valid path $T$ in $G$ from $x$ to $v'$. 

We claim that $T$ is an induced path in $G$. Indeed, $P[x,y]$ is induced by definition, and $L''$ is induced because it is a subpath of the weighted geodesic $S'$. Moreover, by our assumption, $L'' \cap (P[x,y] \cup N(P[x,y])) = \emptyset$. It thus suffices to verify that $(L'' \setminus \{z'\}) \cap (Q' \cup N(Q')) = \emptyset$ and that $z'$ is neither equal to nor adjacent to any vertex of $Q'$ other than $z$. The former follows by the maximal choice of $z'$, while the latter follows by the maximal choice of $z$. Thus, $T$ is an induced path in $G$, as desired.

We now demonstrate that $w(T) > w(P)$. Note that $T$ can be obtained from $P$ by replacing the subpath $Q[z,u] \setminus \{z\}$ (which has weight at most $20W$ by \eqref{eq:boundweightQzu}) with $L'' \setminus \{z\}$ (which has weight at least $w(L'') - W$). (Observe that this includes both possibilities: $z = z'$ or $zz' \in E(G)$.) 
Utilizing \eqref{eq:L-double-prime-large}, we find:
\begin{equation}
\label{eq:wTtoP}
    w(T) \geq w(P) - 20W + w(L'') - W \overset{\eqref{eq:L-double-prime-large}}{\geq} w(P) - 21W + \left(\frac{\Gamma}{2} - 11W\right).
\end{equation}

Given the assumption $\Gamma \ge 300W$, the right-hand side of \eqref{eq:wTtoP} is strictly greater than $w(P)$, demonstrating that $w(T) > w(P)$. Furthermore, $L''$ is a terminal subpath of $T$ that constitutes a weighted geodesic satisfying $w(L'') > \beta$ by \eqref{eq:L-double-prime-large}. As $T$ is an induced path of greater weight than $P$, its existence contradicts the maximality of $P$, thereby proving the claim.
\end{claimproof}

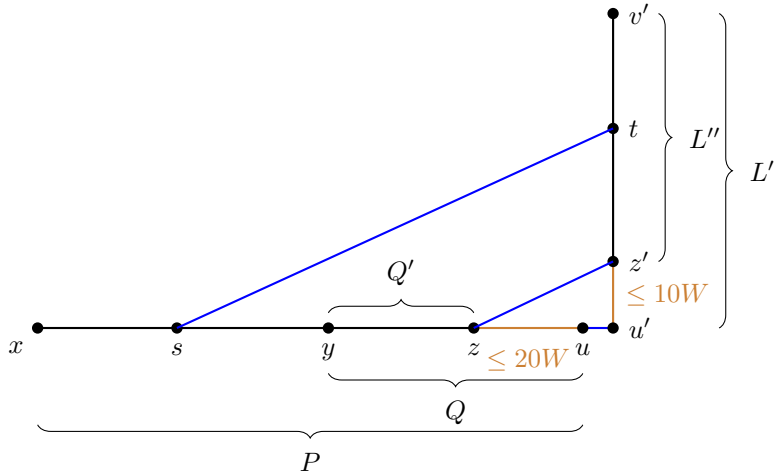
\begin{figure}[htbp]
\centering
\begin{tikzpicture}[scale=0.8,>=stealth]

\tikzset{
    vtx/.style={circle,fill=black,inner sep=1.5pt},
    blueedge/.style={blue,thick},
    brownedge/.style={brown!80!orange,thick},
    brace/.style={decorate,decoration={brace,amplitude=5pt}}
}

\coordinate (X)  at (0,0);
\coordinate (S)  at (2.3,0);
\coordinate (Y)  at (4.8,0);
\coordinate (Z)  at (7.2,0);
\coordinate (U)  at (9.0,0);
\coordinate (Up) at (9.5,0);

\coordinate (Zp) at (9.5,1.1);
\coordinate (T)  at (9.5,3.3);
\coordinate (Vp) at (9.5,5.2);

\draw[thick] (X) -- (Z);
\draw[brownedge] (Z) -- node[below=4pt] {$\leq 20W$} (U);
\draw[blueedge] (U) -- (Up);

\draw[brownedge] (Up) -- node[right=1pt] {$\leq 10W$} (Zp);
\draw[thick] (Zp) -- (Vp);

\node[vtx,label=below left:$x$] at (X) {};
\node[vtx,label=below:$s$] at (S) {};
\node[vtx,label=below:$y$] at (Y) {};
\node[vtx,label=below:$z$] at (Z) {};
\node[vtx,label=below:$u$] at (U) {};
\node[vtx,label=right:$u'$] at (Up) {};

\node[vtx,label=right:$z'$] at (Zp) {};
\node[vtx,label=right:$t$] at (T) {};
\node[vtx,label=right:$v'$] at (Vp) {};

\draw[blueedge]
    (S) -- (T);

\draw[blueedge]
    (Z) --(Zp);

\draw[brace]
    ($(Y)+(0,0.25)$) -- ($(Z)+(0,0.25)$)
    node[midway,above=7pt] {$Q'$};

\draw[brace]
    ($(U)+(0,-0.75)$) -- ($(Y)+(0,-0.75)$)
    node[midway,below=7pt] {$Q$};

\draw[brace]
    ($(U)+(0,-1.6)$) -- ($(X)+(0,-1.6)$)
    node[midway,below=7pt] {$P$};

\draw[brace]
    ($(Vp)+(1.75,0)$) -- ($(Up)+(1.75,0)$)
    node[midway,right=8pt] {$L'$};

\draw[brace]
    ($(Vp)+(0.75,0)$) -- ($(Zp)+(0.75,0)$)
    node[midway,right=8pt] {$L''$};

\end{tikzpicture}
\caption{Construction of the induced cycle $H$ in the proof of Lemma~\ref{lem:weighted-near-transitive}. Blue segments again represent vertex pairs whose distance is either zero or one.}
\label{figure2}
\end{figure}

By Claim~\ref{claim:L''P_meets}, there exist vertices $s\in V(P[x,y])$ and $t\in V(L'')$ such that either $s = t$ or $st \in E(G)$. Choose a pair $(s,t)$ that minimizes the combined weight:
$
    w(P[s,y])+w(L''[z',t]).
$
Recall that either $z=z'$ or $zz' \in E(G)$, and either $s=t$ or $st \in E(G)$. Let $H$ be the subgraph of $G$ defined by the union of the paths $P[s,y]$, $Q[y,z]$, and $L''[z',t]$, along with the edges $zz'$ (if $z \neq z'$) and $st$ (if $s \neq t$). The concatenation of these segments forms a closed walk in $G$.

We claim that $H$ is an induced cycle in $G$. First, because $P$ is an induced path, its subpath $P[s,z]$ is also induced. Similarly, because $L''$ is an induced path (being a subpath of the weighted geodesic $S'$), $L''[z',t]$ is also induced. Furthermore, there are no edges or shared vertices between $Q[y,z]$ and $L''[z',t]$ other than the vertex $z$ (if $z=z'$) or the edge $zz'$. Indeed, this follows because $(L'' \setminus \{z'\}) \cap (Q \cup N(Q)) = \emptyset$ by the maximal choice of $z'$, and because $z'$ is neither equal nor adjacent to any vertex of $Q[y,z]$ other than $z$ by the maximal choice of $z$. Finally, there can be no edges between $P[s,y]$ and $L''[z',t]$ other than $st$ (if $s \neq t$), nor can they share any vertices unless $s=t$. Any such edge or shared vertex would yield a valid pair $(s', t')$ with a strictly smaller combined weight $w(P[s',y]) + w(L''[z',t'])$, directly contradicting the minimal choice of $(s,t)$. Thus, $H$ is an induced cycle, as desired.

It remains to show that $H$ contains a subpath $Q_H$ that is a weighted geodesic satisfying $w(Q_H) \ge \frac{\Gamma}{2}-100W$. We claim that $Q'$ is the desired subpath. Indeed, $Q'$ is a weighted geodesic because it is a subpath of the weighted geodesic $Q$. Moreover, \eqref{eq:Q-prime-large} guarantees that $w(Q') \ge \frac{\Gamma}{2}-100W$, as required. This completes the proof of the lemma.
\end{proof}

\subsection{Putting everything together}
\label{proofsubsectionputtingeverythingtogether}
We are now ready to combine the results of the previous subsections and complete the proof of Theorem~\ref{thm:sqrt-cycle}.

\begin{proof}[Proof of Theorem~\ref{thm:sqrt-cycle}]

Recall that \(G=C(D)\) and that
\(w(F)=|F|\) for each \(F\in V(G)\). Since every vertex of \(G\)
corresponds to a directed cycle in \(D\), we have
$
    W=\max_{F\in V(G)} w(F)=\ell.
$

We begin by showing that
$
    \Gamma_w(G)\le 300\ell.
$
Suppose, for a contradiction, that
$
    \Gamma_w(G)>300\ell.
$
Since Lemma~\ref{lem:nearly-vertex-transitive} guarantees that $G$ is $w$-nearly transitive, Lemma~\ref{lem:weighted-near-transitive} yields an induced cycle $H$ containing a weighted-geodesic subpath $Q_H$ such that
$
    w(Q_H)
    \ge \frac{\Gamma_w(G)}{2}-100\ell
    > 50\ell.
$
On the other hand, Lemma~\ref{lem:weight-le-3ell} gives
$
    w(Q_H)\le 3\ell,
$
a contradiction. Hence, $\Gamma_w(G)\le 300\ell$.

Furthermore, Lemma~\ref{lem:delta-ge-d} implies that $d\le \Gamma_w(G)$.
Combining this with $\Gamma_w(G)\le 300\ell$, we obtain
$
    \ell\ge d/300.
$
Together with Lemma~\ref{lem:small-diameter-bound}, this immediately
implies Theorem~\ref{thm:sqrt-cycle}. Indeed,
$$
\ell
\ge
\max\left\{\frac{n}{9d},\frac{d}{300}\right\}
\ge
\sqrt{\frac{n}{9d}\cdot\frac{d}{300}}
=
\frac{\sqrt n}{30\sqrt 3},
$$
as desired.
\end{proof}

\section{Linear perimeter gap in vertex-transitive digraphs}
\label{sec:perimetergap}

In this section, we prove Theorem~\ref{thm:perimeter-gap}, confirming a conjecture of Buci\'c, Hendrey, Mohar, Steiner and Yepremyan~\cite{BHMSY}. Our strategy is as follows. For any positive integer $m$, we construct a strongly connected, vertex-transitive digraph $D$ on $12m$ vertices such that every directed cycle in $D$ has length divisible by $m$, but $D$ has no Hamiltonian cycle. After formally defining $D$, we prove that it is strongly connected (Lemma~\ref{lem:D_strongly_connected}) and vertex-transitive (Lemma~\ref{lem:D_vertex_trans}). A parity obstruction established in Lemma~\ref{lem:no_di_cycle} rules out a Hamiltonian cycle, and we then verify the divisibility condition, completing the proof of Theorem~\ref{thm:perimeter-gap}. See Section~\ref{proofoverview:improvedupperbound} for an overview of this construction and its parity obstruction.

\subsection{Construction of the digraph \texorpdfstring{$D$}{D}}
\begin{definition}\label{def:directed_D}
Let
$
X=\{(a,b):a,b\in [4],\ a\neq b\}.
$
We define a digraph $D$ with vertex set
$
V(D)=X\times \mathbb{Z}_m.
$
Its directed edges are given by
$$
E(D)=\left\{\big((a,b,t),(b,c,t+1)\big): a,b,c\in [4],\ a\neq b,\ c\notin \{a,b\},\ t\in \mathbb{Z}_m\right\},
$$
where the third coordinate is taken modulo $m$.
\end{definition}

\subsection{Connectivity and vertex-transitivity}

\begin{lemma}\label{lem:D_strongly_connected}
The digraph $D$ is strongly connected.
\end{lemma}

\begin{proof}
Let $B$ be the auxiliary digraph with vertex set $X$ and with an arc
$
(a,b)\to (b,c)
$
whenever $c\notin \{a,b\}$. We first show that $B$ is strongly connected.

The natural action of $S_4$ on $X$ is transitive and preserves the arcs of $B$. Hence, to prove that $B$ is strongly connected, it suffices to show that every vertex of $B$ is reachable from $(1,2)$. Indeed, if $u,v\in X$ and $g\in S_4$ sends $(1,2)$ to $u$, then a directed walk from $(1,2)$ to $g^{-1}(v)$ is sent by $g$ to a directed walk from $u$ to $v$.

We now verify that every vertex of $B$ is reachable from $(1,2)$. In one step, we can reach
$
(2,3),\ (2,4).
$
In two steps, we can reach
$
(3,1),\ (3,4),\ (4,1),\ (4,3).
$
In three steps, we can reach
$
(1,3),\ (1,4),\ (3,2),\ (4,2).
$
Finally,
$
(1,2)\to (2,3)\to (3,4)\to (4,2)\to (2,1)
$
shows that $(2,1)$ is reachable as well. This accounts for all $11$ other vertices in $X$, confirming that $B$ is strongly connected.

Now let
$
(a,b,t), (x,y,s)
$
be arbitrary vertices of $D$. Since $B$ is strongly connected, there is a directed walk in $B$ from $(a,b)$ to $(x,y)$; say it has length $L$. Write this walk as
$
(a_0,a_1)\to (a_1,a_2)\to \cdots \to (a_L,a_{L+1}),
$
where $(a_0,a_1)=(a,b)$ and $(a_L,a_{L+1})=(x,y)$. By the definition of the edges of $D$, for each $i=0,\ldots,L-1$ and $t \in \Z_m$ we have
$
(a_i,a_{i+1},t+i)\to (a_{i+1},a_{i+2},t+i+1),
$
where the third coordinate is taken modulo $m$. Therefore the above walk in $B$ gives a directed walk in $D$ from $(a,b,t)$ to
$
(x,y,t+L).
$

It remains to adjust the length of the walk modulo $m$. We claim that, at every vertex $(a,b)\in X$, the digraph $B$ contains closed directed walks of lengths $3$ and $4$. Indeed, if $c\in [4]\setminus \{a,b\}$, then
$
(a,b)\to (b,c)\to (c,a)\to (a,b)
$
is a closed directed walk of length $3$. Also, if $c,d$ are the two elements of $[4]\setminus \{a,b\}$, then
$
(a,b)\to (b,c)\to (c,d)\to (d,a)\to (a,b)
$
is a closed directed walk of length $4$.

Choose a sufficiently large integer $N$ such that $N \equiv s-t-L \pmod m$. Since every sufficiently large integer is a nonnegative integer combination of $3$ and $4$, we may write
$
N=3r+4q
$
for some nonnegative integers $r,q$. By inserting, at the beginning of the walk from $(a,b)$ to $(x,y)$, $r$ closed directed walks of length $3$ and $q$ closed directed walks of length $4$, all based at $(a,b)$, we obtain a directed walk in $B$ from $(a,b)$ to $(x,y)$ of length $N+L$.

Writing this new walk as
$
(u_0,u_1)\to (u_1,u_2)\to \cdots \to (u_{N+L},u_{N+L+1}),
$
with $(u_0,u_1)=(a,b)$ and $(u_{N+L},u_{N+L+1})=(x,y)$, the definition of $D$ gives a directed walk
$
(u_0,u_1,t)\to (u_1,u_2,t+1)\to \cdots \to (u_{N+L},u_{N+L+1},t+N+L).
$
Since
$
t+N+L\equiv s \pmod m,
$
this is a directed walk in $D$ from $(a,b,t)$ to $(x,y,s)$. Hence every vertex of $D$ can reach every other vertex, so $D$ is strongly connected.
\end{proof}

\begin{lemma}
\label{lem:D_vertex_trans}
The digraph $D$ is vertex-transitive.
\end{lemma}
\begin{proof}
Note that the group $S_4\times \mathbb{Z}_m$ acts on $V(D)$ by
$
(\sigma,r)\cdot (a,b,t)=(\sigma(a),\sigma(b),t+r).
$
This action is transitive on $V(D)$ and preserves directed adjacency. Indeed, if
$
(a,b,t)\to (b,c,t+1)
$
is an edge of $D$, then $c\notin \{a,b\}$, and hence
$
\sigma(c)\notin \{\sigma(a),\sigma(b)\}.
$
Therefore
$
(\sigma(a),\sigma(b),t+r)\to (\sigma(b),\sigma(c),t+r+1)
$
is also an edge of $D$. Thus $D$ is vertex-transitive, as desired.
\end{proof}

\subsection{The parity obstruction for Hamiltonicity}
\label{subsec:parityobstruction}

\begin{lemma}\label{lem:no_di_cycle}
The digraph $D$ has no Hamiltonian cycle.
\end{lemma}

\begin{proof}
Suppose, for contradiction, that $D$ has a Hamiltonian cycle $C$. For each $t\in \mathbb{Z}_m$, let
$
X_t=X\times \{t\}.
$
Since every directed edge of $D$ goes from $X_t$ to $X_{t+1}$, the Hamiltonian cycle $C$ induces a bijection
$
\pi_t:X_t\to X_{t+1}.
$
Identifying each layer $X_t$ with $X$, we regard $\pi_t$ as a permutation of $X$. By the definition of $D$, this permutation has the form
$
\pi_t(a,b)=(b,c)
$
for some $c\notin \{a,b\}$. We now prove that each $\pi_t$ is an even permutation of the $12$-element set $X$. Fix $t$, and write $\pi=\pi_t$. For each $b\in [4]$, define
$
T_b=\{(a,b):a\in [4]\setminus \{b\}\}
$
and
$
H_b=\{(b,c):c\in [4]\setminus \{b\}\}.
$

Since every directed edge originating from a vertex $(a,b)$ must terminate at a vertex of the form $(b,c)$, the image of $T_b$ under $\pi$ is contained in $H_b$; that is, $\pi(T_b) \subseteq H_b$. Because the global map $\pi$ is injective and the finite sets $T_b$ and $H_b$ have the same cardinality ($|T_b|=|H_b|=3$), the restriction $\pi|_{T_b}$ is therefore a bijection from $T_b$ to $H_b$.
Therefore, for each $b\in [4]$, there is a permutation
$
f_b:[4]\setminus \{b\}\to [4]\setminus \{b\}
$
such that
$
\pi(a,b)=(b,f_b(a)).
$
The condition $f_b(a)\notin \{a,b\}$ implies, in particular, that
$
f_b(a)\neq a
$
for every $a\in [4]\setminus \{b\}$. Hence $f_b$ is a fixed-point-free permutation of a $3$-element set, and therefore $f_b$ is a $3$-cycle. In particular, $f_b$ is even.

Define two permutations $F,R:X\to X$ by
$
F(a,b)=(f_b(a),b)
$
and
$
R(a,b)=(b,a).
$
Then
$
\pi=R\circ F.
$
The permutation $F$ is the product of four disjoint $3$-cycles, one on each set $T_b$, and hence $F$ is even. The permutation $R$ is the product of the six disjoint transpositions
$
(a,b)\leftrightarrow (b,a),
$
one for each $2$-element subset $\{a,b\}\subseteq [4]$, and hence $R$ is also even. Therefore $\pi=R\circ F$ is even. Thus $\pi_t$ is even for every $t\in \mathbb{Z}_m$.

Now consider the first-return map on the layer $X_0$ obtained by following the Hamiltonian cycle for $m$ steps:
$
\Pi=\pi_{m-1}\circ \cdots \circ \pi_0.
$
Since each $\pi_t$ is even, the permutation $\Pi$ is even.

On the other hand, since $C$ is a directed cycle containing all $12m$ vertices of $D$, the first-return map $\Pi$ must itself be a single cycle on the $12$ vertices of $X_0$. Indeed, if $\Pi$ had an orbit of length strictly smaller than $12$, then following $C$ from any vertex in that orbit would return to the starting vertex after fewer than $12m$ steps, contradicting the assumption that $C$ is a Hamiltonian cycle of length $12m$. Thus $\Pi$ is a $12$-cycle. Since a $k$-cycle has sign $(-1)^{k-1}$, every $12$-cycle is an odd permutation. This contradicts the fact that $\Pi$ is even. Therefore $D$ has no Hamiltonian cycle.
\end{proof}

\subsection{Proof of Theorem~\ref{thm:perimeter-gap}}

We are now ready to prove Theorem~\ref{thm:perimeter-gap} by combining the results established in the previous subsections.

\begin{proof}
For each positive integer $m$, let $D$ be the digraph on $12m$ vertices defined in Definition~\ref{def:directed_D}. By Lemma~\ref{lem:D_strongly_connected}, $D$ is strongly connected, and hence connected.
Moreover, by Lemma~\ref{lem:D_vertex_trans}, $D$ is vertex-transitive. Finally, every directed edge of $D$ increases the third coordinate by $1$ modulo $m$. Hence, if $C$ is a directed cycle of length $\ell$, then after traversing $C$ the third coordinate has increased by $\ell$ modulo $m$. Since the cycle returns to its starting vertex, we must have
$
\ell\equiv 0 \pmod m.
$
Thus every directed cycle in $D$ has length divisible by $m$.

By Lemma~\ref{lem:no_di_cycle}, $D$ has no directed cycle of length $12m$. Since all directed cycle lengths are divisible by $m$, the longest directed cycle in $D$ has length at most $11m$. Therefore the perimeter gap of $D$ is at least
$
12m-11m=m.
$
Since $n=12m$, this is equal to
$
\frac{n}{12}.
$
As this construction works for every positive integer $m$, the theorem follows for infinitely many values of $n$, namely for all multiples of $12$.
\end{proof}

\section*{Acknowledgements}
We are grateful to Matija Buci\'c and Shoham Letzter for helpful discussions on the Lov\'asz conjecture. 

OpenAI Codex (GPT-5.6 Sol) was used to assist with proofreading this paper. All mathematical ideas and arguments are entirely due to the authors, who take full responsibility for the content.

\printbibliography

\end{document}